\documentclass[12pt]{amsart}

\usepackage{amssymb}
\usepackage{enumerate}
\usepackage{graphicx}
\usepackage{pgf}
 \usepackage{array}
 \usepackage{pstricks}
 \usepackage{pstricks-add}
 \usepackage{pgf,tikz}
 \usetikzlibrary{automata}
 \usetikzlibrary{arrows}
 \usepackage{indentfirst}

\makeatletter
\@namedef{subjclassname@2010}{%
  \textup{2010} Mathematics Subject Classification}
\makeatother
\newtheorem{thm}{Theorem}[section]

\newtheorem{prop}[thm]{Proposition}
\newtheorem{prob}[thm]{Problem}

\theoremstyle{definition}
\newtheorem{defin}[thm]{Definition}

\newtheorem{exa}[thm]{Example}
\newtheorem*{xrem}{Remark}
\numberwithin{equation}{section}

\begin{document}

\baselineskip=17pt

\title[Okounkov bodies]{Numerical and enumerative results on Okounkov bodies}

\author[P. Pokora]{Piotr Pokora}
\address{Institute of Mathematics\\ Pedagogical University of Cracow \\
Podchor\c a\.zych 2, PL-30-084 Krak\'ow, Poland. }
\email{piotrpkr@gmail.com}

\date{}

\begin{abstract}
In this note we focus on three independent problems on Okounkov bodies for projective varieties. The main goal is to present a geometric version of the classical Fujita Approximation Theorem, a Jow-type theorem \cite{Jow} and a cardinality formulae for Minkowski bases on a certain class of smooth projective surfaces.
\end{abstract}



\maketitle
\section{Introduction}
We present three results on Okounkov bodies, mainly for projective surfaces. The first one can be viewed as a geometric Fujita approximation, which tells us that the Fujita Approximation Theorem for big divisors induces the shape approximation of associated Okounkov bodies. The second result is a certain variation on Jow theorem \cite{Jow}, which roughly speaking tells us that Okounkov bodies can be used to check numerical equivalence of pseudoeffective divisors. The last section is devoted to the cardinality problem for \emph{Minkowski bases} \cite{PatDav} for surfaces with rational polyhedral pseudoeffective cones.

Let us recall briefly what Okounkov bodies are. These bodies were introduced independently by Lazarsfeld and Musta\c t\u a \cite{LM09} and Kaveh and Khovanskii \cite{KK} and they are \emph{convex bodies} $\triangle(D) \subset \mathbb{R}^{n}$ attached to big divisors $D$ on smooth projective varieties $X$ of dimension $n$ with respect to an \emph{admissible flag}, i.e.,
a sequence of irreducible subvarieties $X = Y_{0} \supset Y_{1} \supset ... \supset Y_{n} = \{pt\}$ such that ${\rm codim}_{X} Y_{i} = i$ and $Y_{n}$ is a smooth point of each $Y_{i}$'s. We refer to Section 1 in \cite{LM09} for further details about Okounkov bodies.

Recently Okounkov bodies have been applied to some problems appearing in other branches of mathematics, for instance in mathematical physics. One of the most prominent examples is the paper due to Harada and Kaveh \cite{HK} in which the authors consider complete integrable systems in the context of Okounkov bodies. Roughly speaking they showed that the image of the so-called moment map corresponds to a certain Okounkov body, which is highly remarkable.
\section{Geometric Fujita Approximation}
Assume that $X$ is an irreducible complex projective variety of dimension $n > 0$. Recall that for an integral divisor $D$ the \emph{volume} of $D$ is a real number defined by

$${\rm vol}_{X}(D) = {\rm lim \, sup}_{m \rightarrow \infty} \frac{h^{0}(X, \mathcal{O}_{X}(mD))}{m^{n}/n!}.$$

It is well known that $D$ is big if and only if ${\rm vol}_{X}(D) > 0$.

In \cite{LM09} the authors studied the Fujita Approximation Theorem in the language of Okounkov bodies. Let us recall a classical statement of this theorem.

\begin{thm}[Theorem 11.4.4 (Part II), \cite{PAG}]
\label{fujitaapproximation}

Let $D$ be a big integral divisor on $X$ and fix a positive number $\varepsilon >0$. Then there exists a birational morphism $\mu : X' \rightarrow X$, where $X'$ is irreducible, and an integer $p > 0$ such that
$$\mu^{*}(pD) = A + E,$$
where $A$ is an ample divisor and $E$ is an effective divisor, both integral, having the property that
$${\rm vol}_{X'} (A) > p^{n}({\rm vol}_{X}(D) - \varepsilon).$$
\end{thm}
Of course this theorem remains true for numerical classes of divisors. The proof presented in \cite{PAG} is based one the theory of multiplier ideals.

In \cite{LM09} the authors formulated the above result in the language of semigroups and Okounkov bodies.
\begin{thm}[Theorem 3.3, \cite{LM09}]
Let $D$ be a big divisor on $X$ and for numbers $p, k >0$ write
$$V_{k,p} = {\rm Im} \bigg( S^{k} H^{0}(X, \mathcal{O}_{X}(pD)) \rightarrow H^{0}(X, \mathcal{O}_{X}(pkD)) \bigg),$$
where $S^{k}$ denotes the $k$-th symmetric power.
Given $\varepsilon >0$, there exists an integer $p_{0}=p_{0}(\varepsilon)$ having the property that if $p \geq p_{0}$, then
$$\lim_{k \rightarrow \infty} \frac{ \dim V_{k,p}}{p^{n}k^{n}/n!} \geq {\rm vol}_{X}(D) - \varepsilon.$$
\end{thm}
We refer to \cite[Remark 3.4]{LM09} for a link between the classical statement of the Fujita Approximation Theorem with the above result.

Our main aim in this section is to present a certain reformulation of the Fujita Approximation Theorem in the language of shapes of Okounkov bodies for big divisors. Our result tells us that using ample divisors on a modification one can approximate shape of Okounkov bodies of big divisors as precisely as desired. the Fujita Approximation Theorem provides quantitive statement concerning the volume, a numerical invariant of divisors. This article extends this result in a geometrical direction connecting the numerical nature with geometry. Let us point out that a certain kind of approximation can be found also in \cite[Lemma~8]{AKL}.

It may happen that Okounkov bodies for certain ample divisors may not be polyhedral, but there is also a large class of projective varieties and divisors for which Okounkov bodies are rational polyhedral, for instance Okounkov bodies of big divisors on smooth projective surfaces -- see \cite{KLM12} for details and results.

\begin{thm}
\label{Fujita}
Let $X$ be a smooth projective variety of dimension $n$. Assume that $D$ is a big divisor on $X$. Then for every $\beta  > 0$ there exists a birational morphism $\eta: \widetilde{X} \rightarrow X$, an ample divisor $A$ on $\widetilde{X}$, $\delta > 0$ and an admissible flag $Y_{\bullet}$ on $\widetilde{X}$ such that the Okounkov body $\triangle_{Y_{\bullet}}(\eta^{*}(D))$ contains $\triangle_{Y_{\bullet}}(A)$ with
$${\rm vol}_{\mathbb{R}^{n}}(\triangle_{Y_{\bullet}}(\eta^{*}(D)) \setminus \triangle_{Y_{\bullet}}(A)) < \beta$$
and is contained in $\triangle_{Y_{\bullet}}((1+\delta)A)$ with
$${\rm vol}_{\mathbb{R}^{n}}(\triangle_{Y_{\bullet}}((1+ \delta)A) \setminus \triangle_{Y_{\bullet}}(\eta^{*}(D))) < \beta.$$
\end{thm}
Before we proceed to the proof let us present the following picture, which shows geometrical meaning of this theorem (in the case of $n=2$).
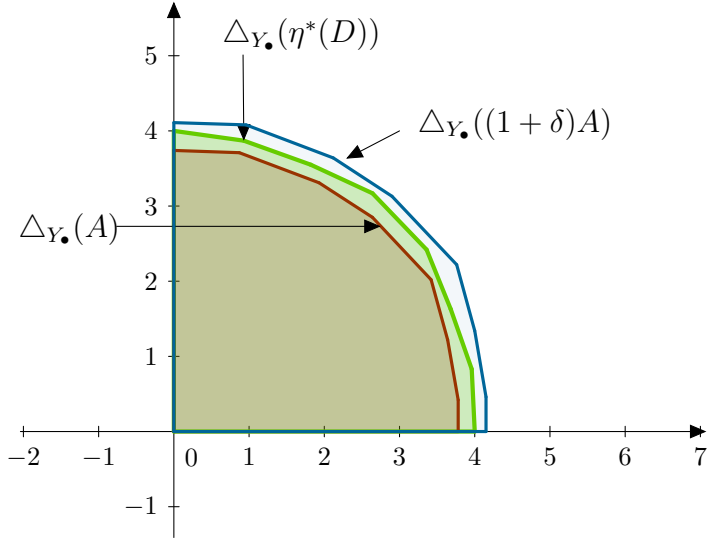
\begin{center}
\definecolor{qqwwzz}{rgb}{0,0.4,0.6}
\definecolor{zzttqq}{rgb}{0.6,0.2,0}
\definecolor{wwccqq}{rgb}{0.4,0.8,0}
\begin{tikzpicture}[line cap=round,line join=round,>=triangle 45,x=1.0cm,y=1.0cm]
\draw[->,color=black] (-2.04,0) -- (7.08,0);
\foreach \x in {-2,-1,1,2,3,4,5,6,7}
\draw[shift={(\x,0)},color=black] (0pt,2pt) -- (0pt,-2pt) node[below] {\footnotesize $\x$};
\draw[->,color=black] (0,-1.41) -- (0,5.72);
\foreach \y in {-1,1,2,3,4,5}
\draw[shift={(0,\y)},color=black] (2pt,0pt) -- (-2pt,0pt) node[left] {\footnotesize $\y$};
\draw[color=black] (0pt,-10pt) node[right] {\footnotesize $0$};
\clip(-2.04,-1.41) rectangle (7.08,5.72);
\fill[line width=1.6pt,color=wwccqq,fill=wwccqq,fill opacity=0.25] (0,0) -- (0,4) -- (0.93,3.87) -- (1.82,3.55) -- (2.64,3.17) -- (3.36,2.42) -- (3.68,1.63) -- (3.96,0.83) -- (4,0) -- cycle;
\fill[line width=1.2pt,color=zzttqq,fill=zzttqq,fill opacity=0.2] (0,0) -- (0,3.74) -- (0.87,3.71) -- (1.93,3.31) -- (2.64,2.85) -- (3.42,2.02) -- (3.64,1.22) -- (3.78,0.42) -- (3.78,0) -- cycle;
\fill[line width=1.2pt,color=qqwwzz,fill=qqwwzz,fill opacity=0.05] (0,0) -- (0,4.11) -- (0.96,4.08) -- (2.12,3.64) -- (2.9,3.13) -- (3.76,2.22) -- (4,1.34) -- (4.15,0.46) -- (4.15,0) -- cycle;
\draw [line width=1.6pt,color=wwccqq] (0,0)-- (0,4);
\draw [line width=1.6pt,color=wwccqq] (0,4)-- (0.93,3.87);
\draw [line width=1.6pt,color=wwccqq] (0.93,3.87)-- (1.82,3.55);
\draw [line width=1.6pt,color=wwccqq] (1.82,3.55)-- (2.64,3.17);
\draw [line width=1.6pt,color=wwccqq] (2.64,3.17)-- (3.36,2.42);
\draw [line width=1.6pt,color=wwccqq] (3.36,2.42)-- (3.68,1.63);
\draw [line width=1.6pt,color=wwccqq] (3.68,1.63)-- (3.96,0.83);
\draw [line width=1.6pt,color=wwccqq] (3.96,0.83)-- (4,0);
\draw [line width=1.6pt,color=wwccqq] (4,0)-- (0,0);
\draw [line width=1.2pt,color=zzttqq] (0,0)-- (0,3.74);
\draw [line width=1.2pt,color=zzttqq] (0,3.74)-- (0.87,3.71);
\draw [line width=1.2pt,color=zzttqq] (0.87,3.71)-- (1.93,3.31);
\draw [line width=1.2pt,color=zzttqq] (1.93,3.31)-- (2.64,2.85);
\draw [line width=1.2pt,color=zzttqq] (2.64,2.85)-- (3.42,2.02);
\draw [line width=1.2pt,color=zzttqq] (3.42,2.02)-- (3.64,1.22);
\draw [line width=1.2pt,color=zzttqq] (3.64,1.22)-- (3.78,0.42);
\draw [line width=1.2pt,color=zzttqq] (3.78,0.42)-- (3.78,0);
\draw [line width=1.2pt,color=zzttqq] (3.78,0)-- (0,0);
\draw [->] (3,4) -- (2.29,3.63);
\draw [->] (-0.76,2.73) -- (2.75,2.73);
\draw (3.1,4.45) node[anchor=north west] {$\triangle_{Y_{\bullet}}((1+ \delta)A)$};
\draw (-2.21,3.05) node[anchor=north west] {$\triangle_{Y_{\bullet}}(A)$};
\draw [->] (0.92,5.01) -- (0.93,3.87);
\draw (0.5,5.64) node[anchor=north west] {$\triangle_{Y_{\bullet}}(\eta^{*}(D))$};
\draw [line width=1.2pt,color=qqwwzz] (0,0)-- (0,4.11);
\draw [line width=1.2pt,color=qqwwzz] (0,4.11)-- (0.96,4.08);
\draw [line width=1.2pt,color=qqwwzz] (0.96,4.08)-- (2.12,3.64);
\draw [line width=1.2pt,color=qqwwzz] (2.12,3.64)-- (2.9,3.13);
\draw [line width=1.2pt,color=qqwwzz] (2.9,3.13)-- (3.76,2.22);
\draw [line width=1.2pt,color=qqwwzz] (3.76,2.22)-- (4,1.34);
\draw [line width=1.2pt,color=qqwwzz] (4,1.34)-- (4.15,0.46);
\draw [line width=1.2pt,color=qqwwzz] (4.15,0.46)-- (4.15,0);
\draw [line width=1.2pt,color=qqwwzz] (4.15,0)-- (0,0);
\end{tikzpicture}
\end{center}

\begin{proof}

By Theorem \ref{fujitaapproximation} we know that for a fixed $\varepsilon > 0$ there exists a birational morphism $\eta: \widetilde{X} \rightarrow X$, an ample divisor $A$ and an effective divisor $E$ on $\widetilde{X}$ such that
$$\eta^{*}(D) = A + E \,\,\,\, {\rm and} \,\,\,\,  {\rm vol}(\eta^{*}(D)) \geq {\rm vol}(A) \geq {\rm vol}(\eta^{*}(D)) - \varepsilon.$$

Take a very general admissible flag $Y_{\bullet}$ on $\widetilde{X}$ satisfying $0 \in \triangle_{Y_{\bullet}}(E)$ and $Y_{1}$ is not contained in the \emph{augmented base locus} $\mathbb{B}_{+}(\eta^{*}(D))$ (see \cite{ELMNP}). These assumptions imply that $\triangle_{Y_{\bullet}}(\eta^{*}(D)) = \triangle_{Y_{\bullet}}(A + E) \supseteq \triangle_{Y_{\bullet}}(A)$ and additionally taking  $\varepsilon$ small enough implies automatically that ${\rm vol}_{\mathbb{R}^{n}}(\triangle(\eta^{*}(D)) \setminus \triangle(A)) < \varepsilon \leq \beta$.

To conclude this proof it is enough to show that there exists $\delta > 0$ as in the theorem. Notice that decreasing the value $\varepsilon$ implies by Theorem \ref{fujitaapproximation} that the Okounkov body $\triangle_{Y_{\bullet}}(A)$ approaches to $\triangle_{Y_{\bullet}}(\eta^{*}D)$ and the difference between volumes of $\triangle(\eta^{*}D)$ and $\triangle(A)$ tends to $0$. Thus for a fixed $\beta > 0$ one can find a sufficiently small $\varepsilon > 0$ and $\delta = \delta(\beta, \varepsilon) > 0$ such that $$\triangle_{Y_{\bullet}}((1+\delta)A) \supset \triangle_{Y_{\bullet}}(\eta^{*}(D))$$ and moreover
$${\rm vol}_{\mathbb{R}^{n}}(\triangle_{Y_{\bullet}}((1 + \delta)A) \setminus \triangle_{Y_{\bullet}}(\eta^{*}(D))) < \beta.$$
This completes the proof.
\end{proof}

\section{Numerical equivalence of pseudoeffective divisors on surfaces}

In \cite{LM09} the authors have showed that Okounkov bodies are both geometrical and numerical in nature, which means that if big divisors $D_{1}, D_{2}$ are numerical equivalent, then $\triangle_{Y_{\bullet}}(D_{1}) = \triangle_{Y_{\bullet}}(D_{2})$ for an admissible flag $Y_{\bullet}$. However, it was not clear whether one can read off all numerical invariants of a given big divisor from its Okounkov bodies with respect to any flag. In \cite{Jow} the author has proved the following very interesting theorem.

\begin{thm}
\label{Jow}
Let $X$ be a normal complex projective variety of dimension $n$. If $D_{1}, D_{2}$ are two big divisors on $X$ such that
$$\triangle_{Y_{\bullet}}(D_{1}) = \triangle_{Y_{\bullet}}(D_{2})$$
for every admissible flag $Y_{\bullet}$ on $X$, then $D_{1}$ and $D_{2}$ are numerically equivalent.
\end{thm}

The proof uses theory of restricted complete linear series and restricted volumes. 
Recently \cite{KL14} the authors have studied another numerical
properties of divisors in the context of Okounkov bodies, i.e.
they presented a certain ampleness and nefness criterion for
divisors on projective surfaces. Our aim is to follow this path by
showing a special version of Jow theorem for complex projective
surfaces. Namely we prove that it is enough to compare only
finitely many Okounkov bodies and possibly infinitely many
intersection numbers (these numbers come from intersections with
irreducible negative curves) in order to obtain the same result.
We will use the following description of Okounkov bodies on
surfaces, which uses Zariski decomposition for
$\mathbb{R}$-pseudoeffective divisors.
\begin{defin}[Zariski decomposition]
\label{ZariskiDec}
   Let $D$ be a pseudo-effective $\mathbb{R}$--divisor on a complex projective surface $Y$.
   Then there exist $\mathbb{R}$--divisors $P_{D}$ and $N_{D}$ such that
   \begin{itemize}
   \item[a)] $D=P_{D}+N_{D}$;
   \item[b)] $P_D$ is a nef divisor and $N_D$ is either empty or supported
      on a union of curves $N_{1},\ldots,N_{r}$ with negative definite
      intersection matrix;
   \item[c)] $N_{i}.P_{D} = 0$ for each $i=1, \ldots, r$.
   \end{itemize}
\end{defin}
We refer to \cite[Chapter 14]{Badescu} for a nice expository presentation of the notion of Fujita-Zariski decomposition for $\mathbb{R}$-divisors.
In the sequel we will use the following description of Okounkov bodies for smooth projective surfaces.
\begin{thm}[Theorem 6.4, \cite{LM09}]
\label{thm:okounkov bodies for surfaces}
Let $D$ be a big  $\mathbb{Q}$-divisor on a smooth complex projective surface $Y$ and let $(x, C)$ be an admissible flag. Suppose that $C$ is not contained in $\mathbb{B}_{+}(D)$. Let $a$ be the coefficient of $C$ in the negative part of the Zariski decomposition. For $t \in [a, \mu]$ let us define $D_{t} = D - tC$, where $0 \leq a \leq \mu$. Consider $D_{t} = P_{t} + N_{t}$ the Zariski decomposition of $D_{t}$. Put
$$\alpha (t) = {\rm ord}_{x}(N_{t}), \,\,\,\, \beta(t) = \alpha(t) + {\rm vol}_{X|C}(P_{t}) = {\rm ord}_{x}(N_{t}) + P_{t}.C.$$
Then the Okounkov body of $D$ is the region bounded by the graphs of $\alpha$ and $\beta$, i.e.
$$\triangle(D) = \{(t, y) \in \mathbb{R}^2 : a \leq t \leq \mu \,\,\,  \wedge \,\,\, \alpha(t) \leq y \leq \beta(t)\}.$$
Moreover, $\alpha$ and $\beta$ are piecewise linear functions with rational slopes, $\alpha$ is convex and increasing, $\beta$ is concave.
\end{thm}
Let us point out that in \cite{KLM12} the authors showed that in fact Okounkov bodies for surfaces are rational polyhedrons described by almost rational data -- see \cite{KLM12} for details.

Now we present our approach to Theorem \ref{Jow} for projective surfaces.
\begin{prop} Let $Y$ be a smooth complex projective surface. Denote by $\rho$ the Picard number of $Y$. Then there exists a set of irreducible ample divisors $\{A_{1}, ..., A_{\rho}\}$ and a set of very general points $\{x_{1}, ..., x_{\rho}\}$ with $x_{i} \in A_{i}$, such that for two big $\mathbb{R}$-divisors $D_{1}$, $D_{2}$ if
$$\triangle_{(x_{i},A_{i})}(D_{1}) = \triangle_{(x_{i},A_{i})}(D_{2})$$ for every $i \in \{1, ..., \rho\}$, then the positive parts of the Zariski decompositions $P_{1}, P_{2}$ of $D_{1}, D_{2}$ are numerical equivalent.
\end{prop}

\begin{proof}  Let us choose an ample base $\mathcal{B} = \{A_{1}, ..., A_{\rho}\}$ for $N^{1}(Y)$. Without loss of generality we may assume that $A_{1}, ..., A_{\rho}$ are effective and let us choose irreducible curves $C_{i} \in |A_{i}|$ for $i \in \{1, ..., \rho\}$.

Fix a flag $(x_{i}, C_{i})$. Thus by Theorem \ref{thm:okounkov bodies for surfaces} we have that for a big divisor~$D$
$$\triangle(D) = \{ (t, y) \in \mathbb{R}^2 : 0 \leq t \leq \mu \,\, \& \,\, \alpha(t) \leq y \leq \beta(t)\}.$$
Since $x_{i}$ is a very general point, thus $\alpha(t) \equiv 0$ and $\beta(t) = P_{t}.C_{i}$, where $P_{t}$ is the positive part of the Zariski decomposition of $D_{t} = D-tC_{i}$.
Combining this with the condition $\triangle_{(x_{i}, C_{i})}(D_{1})=\triangle_{(x_{i},C_{i})}(D_{2})$ one obtains that $P_{1}$ and $P_{2}$ are numerical equivalent, which ends the proof.
\end{proof}
In order to finish this construction it is enough to construct a \emph{test configuration} for negative parts $N_{1}$ and $N_{2}$. We need to check additionally intersections with all irreducible negative curves $\mathcal{I}(Y) = \{C_{j}\}_{j\in I}$ on $Y$. Then $C_{j}.N_{1} = C_{j}.N_{2}$ for every $j \in I$ implies that $N_{1}$ and $N_{2}$ are numerical equivalent. These considerations lead us to the following result.
\begin{thm} Let $Y$ be a smooth complex projective surface. Denote by $\rho$ the Picard number of $Y$. Assume that $D_{1}, D_{2}$ are $\mathbb{R}$-pseudoeffective divisors and let $D_{j} = P_{j} + N_{j}$ be the Zariski decompositions. There exist irreducible ample divisors $A_{1}, ..., A_{\rho}$ with general points $x_{i} \in A_{i}$, such that $D_{1}, D_{2}$ are numerical equivalent if and only if
\begin{itemize}
\item $\triangle_{(x_{i},A_{i})}(D_{1}) = \triangle_{(x_{i},A_{i})}(D_{2})$ for every $i \in \{1, ..., \rho\}$,
\item $C.N_{1} = C.N_{2}$ for every negative curve $C \in \mathcal{I}(Y)$.
\end{itemize}
\end{thm}
\begin{xrem}
If $Y$ is a smooth complex projective surface with the rational polyhedral pseudoeffective cone, then there is only finitely many negative curves and this implies that in order to check numerical equivalence of two pseudoeffective divisors $D_{1}, D_{2}$ it is enough to compare only finitely many data.
\end{xrem}

\section{On the cardinality of Minkowski bases}
In this section we present a formulae to compute cardinalities of Minkowski bases for a certain class of projective surfaces. The idea of a Minkowski decomposition was presented in \cite{Pat}, where the author studied this concept for the blow-up of $\mathbb{P}^{2}$ at three non-collinear points. Basically the idea of a Minkowski decomposition is the following. Assume (for whole this section) that $Y$ is a smooth projective surface with the rational polyhedral pseudoeffective cone. Let $(x, C)$ be a flag such that $C$ is a big and nef curve with a general point $x \in C$. Then one can construct the set of nef divisors ${\rm MB}_{(x,C)} = \{M_{1}, ..., M_{k}\}$ such that for any big and nef $\mathbb{R}$-divisor $D$ one has
$$D = \sum_{i} \alpha_{i}M_{i} \,\,\,\, \text{ and } \,\,\,\, \triangle(D) = \sum_{i} \alpha_{i} \triangle(M_{i}),$$
where the second sum is the \emph{Minkowski sum} of convex bodies
$$A + B = \{a+b, a \in A \,\, \& \,\, b \in B\}.$$
Elements $M_{i}$ can be viewed as \emph{building blocks} and it can be shown that these blocks are simplicial.
The natural question is the following.
\begin{prob}
Let $(x, C)$ be a fixed flag. What is the cardinality of the Minkowski basis ${\rm MB}_{(x,C)}$?
\end{prob}
In \cite{PSz} we showed that if $C$ is an ample curve, then the cardinality is maximal possible, which means in other words that such flags deliver the largest possible complexity of computations. On the other hand, it seems to be reasonable to ask about a \emph{minimal Minkowski basis}, i.e. a basis for which the number of Minkowski basis elements is the smallest possible. Before we recall our result for ample flags let us introduce some notions.

Suppose that $P$ is a big and nef divisor. Then the \emph{Zariski chamber} associated to $P$ is defined as
\begin{multline*}
\Sigma_{P} = \{B \in {\rm Big}(Y) : \text{ irreducible components of } N_{B}  \\
  \text{ are the only irreducible curves on $Y$ that intersect} P \text{ with multiplicity } 0 \}.
\end{multline*}
By Theorem 1.3 in \cite{BKS04} we know that Zariski chambers yield a locally finite decomposition on the cone ${\rm Big}(X)$ into locally polyhedral subcones such that the support of the negative part of Zariski decompositions of all divisors in the subcone is constant.

The construction of Minkowski bases \cite{PatDav} tells us that for every Zariski chamber one assigns the associated Minkowski basis element, which is by the construction nef. Now we present the idea how to find Minkowski basis elements. Let $\Sigma$ be a Zariski chamber with the support ${\rm Neg}(\Sigma) = \{N_{1}, ..., N_{r}\}$. Then by \cite[Section~3.1]{PatDav} we have $M_{\Sigma}= dC + \sum_{i=1}^{r} a_{i}N_{i}$ with real coefficients $a_{i}$, which are the solution of the following system of equations
   \begin{equation}
   \label{minkowskicoefficients}
   S(a_{1}, ..., a_{r})^{T} = -d(C.N_{1}, ..., N_{r})^{T}.
   \end{equation}
By $S$ we mean the $r \times r$ intersection matrix of negative curves $N_{1}, ..., N_{r}$. As we can see the construction of a Minkowski decomposition relies on the full description of negative curves, which determine Zariski chambers.

We define two numbers
$${\rm NnB}(Y) = \# \{ D \in N^{1}(Y) : D \text{ is nef and not big }\},$$
$${\rm Zar}(Y) = \text{number of Zariski chambers except nef cone}.$$

\begin{thm}[Theorem 3.3, \cite{PSz}]
   Let $Y$ be a smooth complex projective surface with $\overline{{\rm Eff}(Y)}$ rational polyhedral.
   Given a flag $(x, A)$, where $A$ is an ample curve and $x$ is a smooth point on $A$, there is
   $$\# {\rm MB}(x,A) = 1 + {\rm NnB}(Y) + {\rm Zar}(Y).$$
\end{thm}
It is worth to point out that the cardinality of a Minkowski basis with respect to an ample flag can be computed directly from the shape of the nef cone, see \cite{PSz} for details. In order to reduce complexity of computations it is natural to consider a case when $C$ is not any ample curve. For a big and nef divisor (not ample) $C$ let us define
$${\rm NZ}(C) = \# \{\Sigma : {\rm Neg}(\Sigma) \cap {\rm Null}(C) \neq \emptyset \}.$$
Notice that ${\rm NZ}(C) > 0$.


\begin{thm} Let $Y$ be a smooth projective surface which contains only finitely many negative curves satisfying the following condition
\begin{center}
$(\star)$ if two irreducible distinct negative curves $N_{1}, N_{2}$ meet, then $N_{1}.N_{2} \geq \sqrt{N_{1}^{2}N_{2}^{2}}.$
\end{center}
Let $(x, C)$ be an admissible flag, where $C$ is big and nef. Then
$$\# {\rm MB}_{(x,C)} = 1 + {\rm NnB}(Y) + {\rm Zar}(Y) - {\rm NZ}(C).$$
\end{thm}

\begin{proof}
Since the number $1 + {\rm NnB}(Y)$ is fixed for every projective surface (does not depend on an admissible flag), thus we need to compute the number ${\rm Zar}(Y) - {\rm NZ}(C)$.
Notice \cite[Theorem 3]{RSz} that the condition $(\star)$ tells us that for every Zariski chamber $\Sigma$ if ${\rm Neg}(\Sigma) = \{N_{1}, ..., N_{k}\}$ is the support of the negative part of Zariski decomposition, then the associated intersection matrix $S = [N_{i}.N_{j}] \in M_{k \times k}(\mathbb{Z})$ is diagonal.

Suppose that ${\rm Neg}(\Sigma) = \{N_{1}, ..., N_{k}\}$ and let $M_{\Sigma} = dC + \sum_{j=1}^{k} a_{j}N_{j}$ be a Minkowski basis element with fixed $d \neq 0$. 
Assume that $N_{s+1}, .., N_{k} \in {\rm Null}(C)$ and $N_{1}, ..., N_{s} \not\in {\rm Null}(C)$. By the construction of Minkowski basis elements we have $M_{\Sigma} \in {\rm Neg}(\Sigma)^{\perp}$. This implies that for every $N_{i} \in {\rm Null}(C)$ one has
$$0 = M_{\Sigma}.N_{i} = \sum_{j}a_{j}N_{j}.N_{i} = a_{i}N_{i}^{2},$$
and $a_{i} = 0$. We obtain
$$M_{\Sigma} = dC + \sum_{j=1}^{s}a_{j}N_{j}$$
with $a_{j} > 0$. Since for all such surfaces the intersection matrix of curves in the negative part of Zariski decompositions is $ - {\rm diag}(\lambda_{1}, ..., \lambda_{r})$ with $r = \# {\rm Neg}(\Sigma')$ and $\lambda_{j} \geq 1$, thus the corresponding intersection matrix is $-{\rm diag}(\lambda_{1}, ..., \lambda_{s})$. On the other hand, \cite[Proposition 1.1]{BFN10} tells us that there is one to one correspondence between sets of reduced curves which have negative definite intersection matrix with Zariski chambers, thus a diagonal matrix $-{\rm diag}(\lambda_{1}, ..., \lambda_{s})$ corresponds to another Zariski chamber. This completes the proof.
\end{proof}

By \cite{RSz} we know that for a projective surface which satisfies the condition $(\star)$ all Zariski chambers are \emph{simple Weyl chambers}, which means that \emph{Zariski chambers are determined by intersections}. However, it is not clear how a general formulae for the cardinality of Minkowski bases should look like when the condition $(\star)$ does not hold since Minkowski basis elements are \emph{determined by intersection matrices}. One can obviously give the following upper-bound. Let $C$ be a big and nef curve. Using curves in ${\rm Null}(C)$ one can find ${\rm NullZar}(X) \geq 1$ subsets such that corresponding intersection matrices of these curves are negative definite -- it means that these sets are supports of the negative parts of Zariski chambers. Then we have
\begin{equation}
\label{eq:nierownoscmocy}
\# {\rm MB}_{(x,C)} \leq 1 + {\rm NnB}(X) + {\rm Zar}(X) - {\rm NullZar}(X).
\end{equation}
Now we present an example, which shows that the above bound is sharp.
\begin{exa} This construction comes from \cite{BauerFunke}. Let $Y$ be a smooth quartic surface in $\mathbb{P}^3$, which contains a hyperplane section that decomposes into two lines $L_{1}, L_{2}$ and an irreducible conic $C$. The existence of such surfaces was proved for instance in \cite[Lemma 2.2B]{Bauer1}. Such surface has the Picard number $3$ and the pseudoeffective cone is generated by $L_{1}, L_{2}$ and $C$. Curves $L_{1}, L_{2}, C$ have the following intersection matrix
$$\left(
                              \begin{array}{ccc}
                               -2 & 1  & 2 \\
                                1 & -2 & 2 \\
                                2 & 2  &-2 \\
                              \end{array}
                            \right).$$
The BKS decomposition consists of five chambers, namely the nef chamber, which is spanned by $\{L_{1} + C, L_{2} + C, C+ 2L_{1} + 2L_{2}\}$, one chamber corresponding to each $(-2)$-curve and one chamber with support $\{L_{1}, L_{2}\}$.
Fix the flag $(x,D)$, where $D \in |C + 2L_{1} + 2L_{2}|$ and $x \in D$ is a general point. Of course $D$ is big. Simple computations shows that
$$(C+2L_{1} + 2L_{2}).L_{1} = CL_{1} + 2L_{1}L_{1} +2L_{1}L_{2} = 2 - 4 + 2 = 0,$$
$$(C + 2L_{1} + 2L_{2}).L_{2} = 0,$$
thus
$${\rm Null}(C+2L_{1}+2L_{2}) = \{L_{1}, L_{2}\}.$$
Zariski chambers corresponding to $\{L_{1}\}$, $\{L_{2}\}$, $\{L_{1}, L_{2}\}$ have the same Minkowski basis element $D$.
Since $(C + 2L_{1} + 2L_{2}).C = -2 + 4 + 4 = 6$, thus by the construction of Minkowski basis elements one has
$$M = C + 2L_{1} + 2L_{2} + 3C = 4C + 2L_{1} + 2L_{2}.$$
Summarizing up all computations, the Minkowski basis with respect to the flag $(x,D)$ is
$${\rm MB}_{(x,D)} = \{C, L_{1} + C, L_{2} + C, 4C + 2L_{1} + 2L_{2}\},$$
and the number of elements is equal to
$$1 + {\rm NnB}(Y) + {\rm Zar}(Y) - {\rm NZ}(Y) = 4.$$
\end{exa}


Before we end this note let us point out that the cardinality of a Minkowski basis can be computed (quite) efficiently using a computer programme. In \cite{BSch12}, \cite{BFN10} the authors presented a certain backtracking algorithm, which allows to find Zariski chambers just by working on the intersection matrix of all negative reduced curves. A slightly modified version of this algorithm allows us to check whether all Zariski chambers are determined by intersections (we need to check that all negative definite principal submatrices of the intersection matrix of the negative curves are diagonal matrices).

In order to compute cardinalities of Minowski bases it is enough to proceed almost along the same lines. Suppose that $Y$ has only Zariski chambers determined by intersections. If $C$ is a big and nef curve, then $\mathbb{B}_{+}(C)$ is supported on negative curves. Now we can consider the submatrix $N$ of the intersection matrix of all negative curves $M$ -- we remove all rows and columns which correspond to negative curves from $\mathbb{B}_{+}(C)$. It is easy to see that the collection of Zariski chambers determined by $N$ is a subcollection of Zariski chambers determined by $M$ and moreover the number of all Zariski chambers given by $N$ is equal to ${\rm Zar}(Y) - {\rm NZ}(C)$.

\subsection*{Acknowledgements}
The author would like to express his gratitude to Alex K\"uronya, Victor Lozovanu, Tomasz Szemberg, David Schmitz and Stefano Urbinati for useful remarks and to Jarek Buczy\'nski for useful comments which improved the exposition of Theorem 2.3 and Dave Anderson for pointing out Lemma 8 in \cite{AKL}. The author is partially supported by NCN Grant 2014/15/N/ST1/02102.

\end{document}